\documentclass[11pt]{amsart}
\usepackage{amssymb,amscd,amsmath,enumerate}
\usepackage[all]{xy}
\newtheorem{thm}{Theorem}[section]

\newtheorem{cor}[thm]{Corollary}

\newtheorem{prop}[thm]{Proposition}

\newtheorem{lem}[thm]{Lemma}

\theoremstyle{definition}

\numberwithin{equation}{section}

\newcommand{\Irr}{\mathrm{Irr}\,}
\newcommand{\Irra}{\mathrm{Irr}}

\newcommand{\F}{{\mathbb{F}}}

\newcommand{\Z}{{\mathbb{Z}}}

\newcommand{\uG}{{\underline{G}}}

\newcommand{\uX}{{\underline{X}}}
\newcommand{\uY}{{\underline{Y}}}

\newcommand{\cF}{{\mathcal{F}}}

\newcommand{\End}{\mathrm{End}}

\newcommand{\GL}{\mathrm{GL}}
\newcommand{\Grass}{\mathrm{Grass}}
\newcommand{\SL}{\mathrm{SL}}

\newcommand{\PGL}{\mathrm{PGL}}

\newcommand{\Span}{\mathrm{Span}}

\newcommand{\Spec}{\mathrm{Spec}\;}

\newcommand{\im}{\mathrm{im}}

\newcommand\emptypart{\emptyset}

\begin{document}
\title{Flatness of the commutator map over $\SL_n$}

\author{Michael Larsen}
\email{mjlarsen@indiana.edu}
\address{Department of Mathematics\\
    Indiana University \\
    Bloomington, IN 47405\\
    U.S.A.}

\author{Zhipeng Lu}
\email{zhipeng.lu@uni-goettingen.de}
\address{Mathematisches Institut\\
Georg-August-Universität G\"ottingen
Wilhelmsplatz 1
37073 G\"ottingen
Germany}

\begin{abstract}
Let $K$ be any field and $n$ a positive integer.  If we denote by $\xi_{\SL_n}\colon \SL_n\times \SL_n\to \SL_n$
the commutator morphism over $K$, then $\xi_{\SL_n}$ is flat over the complement of the center
of $\SL_n$.
\end{abstract}

\thanks{ML was partially supported by NSF grant DMS-1702152.}

\maketitle

\section{Introduction}

Let $\uG$ denote a semisimple algebraic group over a field $K$, and let $g$ be a positive integer.  We define $\xi_{\uG,g}\colon \uG^{2g}\to \uG$
to be the morphism of varieties defined by the genus $g$ surface word:
$$\xi_{\uG,g}(x_1,y_1,\ldots,x_g,y_g) = [x_1,y_1]\cdots[x_g,y_g],$$
where $[x,y] := xyx^{-1}y^{-1}$.  It is known that for any $\uG$ and any $g\ge 2$, this morphism is flat.
To the best of our knowledge, this observation was first made explicitly in a 2016 paper of Avraham Aizenbud and Nir Avni \cite{AA}, but it may have been known quite a bit earlier.
In \cite{AA}, it is credited to Jun Li \cite{Li}, and it certainly follows easily from estimates of Martin Liebeck and Aner Shalev \cite{LS}.
On the other hand, it is easy to see that for $g=1$, the morphism $\xi_{\uG} := \xi_{\uG,1}$ cannot be flat, since generically $\dim \xi_{\uG}^{-1}(g) = \dim \uG$, while $\dim \xi_{\uG}^{-1}(1) = \dim \uG+\mathrm{rank}\,\uG$.  
One might still hope that $\xi_{\uG}$ is flat away from $\xi_{\uG}^{-1}(1)$.  In this paper, we prove that this is true for $\SL_n$ when $n$ is prime and moreover that for all $n$,
$\xi_{\SL_n}$ is flat away from the inverse image of the center of $\SL_n$.

By a standard argument which can be carried out in the language of model theory or that of algebraic geometry, it suffices to prove the statement when $K$ is a finite field.
Since $\SL_n^2$ and $\SL_n$ are non-singular, it suffices to prove that the dimensions of all fibers over non-central elements are equal.
Using Lang-Weil, we convert this to a counting problem, and we can then use the Frobenius formula to express the commutator fibers in terms of the characters
of $\SL_n(\F_q)$.  It is technically easier to work with the characters of $\GL_n(\F_q)$ (where the character table was first calculated by J.\ A.\ Green \cite{Green}),
so most of this paper is devoted to character estimates on groups of this kind.  A key role is played by an estimate of Roman Bezrukavnikov, Martin Liebeck, Aner Shalev, and Pham Tiep \cite{BLST}.

Let $q$ denote a prime power.  We define $G_n$ to be $\GL_n(\F_q)$, and we denote by $V := V_n$ the vector space $\F_q^n$ of column vectors on which $G_n$ acts.  Throughout this paper, the implicit constants in estimates of the form $O(q^m)$ will always be understood to depend on $n$ but not on $q$ (possible dependence on other parameters will be mentioned explicitly).

Let $\xi_{G_n}\colon G_n\times G_n\to G_n$ denote the commutator map at the level of sets.  A well-known theorem of Frobenius implies that
\begin{equation}
\label{Frob}
|\xi_{G_n}^{-1}(g)| = |G_n|\sum_{\chi\in \Irr G_n} \frac{\chi(g)}{\chi(1)}.
\end{equation}

\section{General Character Estimates for $\GL_n(\F_q)$}

The goal in this section and the next is to prove upper bounds for values of irreducible characters $\chi$ of $G_n$ at non-central elements $g$ sufficient to bound the Frobenius sum
$$\sum_{\chi\in \Irr G_n} \frac{\chi(g)}{\chi(1)}.$$

We begin by recalling the classification of conjugacy classes and irreducible characters
of $G_n$.  Our general reference for this section is \cite{Green}.

Let $\Lambda$ denote the set of all partitions, which we consider as the free multiplicative monoid on generators $1,2,3,\ldots$.  We denote the empty partition $\emptypart$.  For $\lambda\in \Lambda$, we denote by $|\lambda|$ the sum of the parts of $\lambda$.  We denote by $\Lambda_n$ the set of $\lambda\in \Lambda$ with $|\lambda|=n$.

We define a \emph{type} to be a function $\tau\colon \Lambda\setminus\{\emptypart\}\to \Lambda$ which is \emph{finitely supported}, i.e., satisfying 
$$|\Lambda\setminus \tau^{-1}(\emptypart)| < \infty.$$
The \emph{degree} of the type is given by the formula
\begin{equation}
\label{weight}
n = \sum_{\lambda\neq\emptypart} |\lambda||\tau(\lambda)|.
\end{equation}
There is a natural multiplication on the set of types, and degrees add.  Every type has a unique decomposition as a product of primary types.

The set of types of degree $n$ will be denoted $T_n$.  For instance, $T_2$ consists of four types, each supported on a single element of $\Lambda\setminus\{\emptypart\}$: 
\begin{equation}
\label{two}
2\mapsto 1,\ 1\mapsto 1^2,\ 1\mapsto 2,\ 1^2\mapsto 1.
\end{equation}
In general, $T_n$ is finite.
A type $\tau$ is \emph{primary} if its support consists of a single $\lambda$ and $\tau(\lambda)$ is a generator $s$ of $\Lambda$.

We denote by $F$ the set of monic irreducible polynomials in $\F_q[t]$, excluding the polynomial $t$; for $f\in F$, $d(f)$ denotes the degree of $f$.  
For $n\ge 1$, Jordan decomposition gives a natural bijection between finitely supported functions $\varphi\colon F\to \Lambda$ for which 
$$deg\varphi := \sum_{f\in F} d(f) |\varphi(f)| = n$$
and conjugacy classes in $G_n$.  Note that in the notation of \cite{Green}, a single Jordan block of size $k$ corresponds to the partition $1^k$, while $k$ distinct blocks of size $1$ correspond to $k$.

We say \emph{$\varphi$ has type $\tau$} if 
for all $\lambda\neq\emptypart$ we have $\tau(\lambda) = \prod_{f\in \varphi^{-1}(\lambda)} d(f)$.
(Note that this product takes place in $\Lambda$, so that, e.g., $1\cdot 1 = 1^2$.)
For example, a conjugacy class of $G_2$ belongs to one of the four types in $T_2$ listed in (\ref{two})
if, respectively, it is central, split regular semisimple, non-split regular semisimple, or a scalar multiple of the class of transvections.
The conjugacy class associated to $\varphi$ is said to be primary if $c$ is of primary type; equivalently, if the characteristic polynomial is a power of an irreducible polynomial $F$.

The situation for $\Irr G_n$ is dual.  We define an \emph{$s$-simplex} to be an $s$-element subset of $\Z/(q^s-1)\Z$ which forms a single orbit under the action on $\Z/(q^s-1)\Z$
of the group generated by multiplication by $q$.  An $s$-simplex can therefore be represented as 
$$\{k,qk,q^2k,\ldots,q^{s-1}k\}$$ 
where $k\in \Z/(q^s-1)\Z$, and $(q^t-1)k\neq 0$ for every proper divisor $t$ of $s$.
We denote by $\Sigma$ the set of all simplices, i.e., the disjoint union over positive integers $s$ of the set of $s$-simplices.
If $\sigma$ is an $s$-simplex, we call $s$ the \emph{degree} of $\sigma$ and denote it $d(\sigma)$.  The irreducible representations of $G_n$ are in bijective correspondence with functions $\psi\colon \Sigma\to \Lambda$ satisfying
$$\sum_{\sigma\in \Sigma} d(\sigma)|\psi(\sigma)| = n.$$

The character associated to $\psi$ has type $\tau$ if 
for all $\lambda\neq\emptypart$ we have $\tau(\lambda) = \prod_{\sigma\in \psi^{-1}(\lambda)} d(\sigma)$.
A character of $G_2$ belongs to one of the four types in $T_2$ listed in (\ref{two})
if, respectively, it is linear, principal series, discrete series, or (up to twist) Steinberg.
A character is \emph{primary} if it is associated with a single simplex $\sigma$, i.e., if $\psi$ is of primary type.
Following \cite{Green}, we denote by $(-1)^{n-v}I_s^k[\lambda]$ the primary character associated to the $s$-simplex $\sigma = \{k,qk,\ldots,q^{s-1}k\}$ and the partition $\lambda$ of $v := n/s$.  
In \S3, we estimate the values of characters of this form.


For $\tau\in T_n$, we denote by $\Irra_\tau G_n$ the set of characters of type $\tau$.
There may or may not exist $\tau\in T_n$ such that $\Irra_\tau G_n$ is  empty, but if $q$ is sufficiently large in terms of $n$, there exist characters of all types.
We have
$$\sum_{\chi\in \Irr G_n} \frac{\chi(g)}{\chi(1)} = \sum_{\tau\in T_n} \sum_{\chi\in \Irra_\tau G_n}\frac{\chi(g)}{\chi(1)}.$$
The degree of a character depends only on its type $\tau$ \cite[Theorem 14]{Green}; we denote this common degree $\deg \tau$.

If $n_1+\cdots+n_m = n$ and $\alpha_i\in \Irr G_{n_i}$, we denote by $\alpha_1\circ \cdots\circ \alpha_m$ the character of $G_n$ obtained by parabolic induction of $\alpha_1\boxtimes\cdots\boxtimes \alpha_m$.
By \cite[Theorem 2]{Green},
\begin{equation}
\label{Hall}
\alpha_1\circ\cdots\circ\alpha_m(g) = \sum_{c_1,\ldots,c_m} H_{c_1,\ldots,c_m}^g\alpha_1(c_1)\cdots\alpha_m(c_m),
\end{equation}
where $c_i$ denotes a conjugacy class of $G_{n_i}$, and $H_{c_1,\ldots,c_m}^g$ is the number of flags $V^\bullet =  (V^1\supset V^2\supset \cdots\supset V^m = \{0\})$ fixed by $g$ such that $\dim V^{i-1}/V^i = n_i$
and the conjugacy class of $g_i$ in $\GL(V^{i-1}/V^i)$ is $c_i$.

If $\sigma_1,\ldots,\sigma_k$ are pairwise distinct simplices and $\chi_1,\ldots,\chi_m$ are primary characters associated to simplices $\sigma_1,\ldots,\sigma_m$ and partitions $\lambda_1,\ldots,\lambda_m$ respectively, then
$\chi_1\circ\cdots\circ\chi_m$
is the irreducible character of $G_n$ associated with the function $\psi\colon \Sigma\to \Lambda$ such that $\psi(\sigma_i) = \lambda_i$ and $\psi(\sigma) = \emptypart$ if $\sigma\not\in \{\sigma_1,\ldots,\sigma_m\}$ \cite[Theorem 13]{Green}.
In particular, the type of $\chi_1\circ\cdots\circ\chi_m$ is the product of the types of the $\chi_i$. 

In order to use equation (\ref{Hall}), we will need to estimate the number of flags of $V$ respecting the action of $g$.
If $V^\bullet$ denotes a flag in $V$, $g$ is any element of $\End_{\F_q}(V) = M_n(\F_q)$, and $S$ is a subset of the index set $\{1,2,\ldots,m\}$, we say that $V^\bullet$ is \emph{$g$-stable with respect to $S$} if $g$ stabilizes each $V^i$ and $g$ acts as a scalar on $V^{i-1}/V^i$ for all $i\not\in S$.  It is clear that if $V^\bullet$ is $g$-stable then it is $P(g)$-stable for every polynomial $P(t)\in \F_q[t]$; in particular, it is both  $g_s$-stable and $g_n$-stable, where
$g = g_s + g_n$ is the additive Jordan decomposition.
We say $V^\bullet$ is \emph{strictly $g$-stable with respect to $S$} if it is $g$-stable and additionally, $g$ does not act as a scalar on $V^{i-1}/V^i$ for all $i\in S$.

\begin{prop}
\label{flag-prob}
Let  $m\ge 2$, and let $n>a_1 > a_2 > \cdots > a_m = 0$ be a fixed strictly decreasing sequence of positive integers, and $S$ a subset of the index set $\{1,2,\ldots,m\}$.

Let $g$ be any non-central element of $M_n(\F_q)$.  The probability that a uniformly randomly chosen flag 
$$V^1\supset V^2\supset\cdots \supset V^m = \{0\}$$ 
of subspaces of $V^0 := V$ of dimensions $a_1,\ldots,a_m$ is $g$-stable with respect to $S$ is
\begin{equation}
\label{flag-est}
O\bigl(q^{1-n+\sum_{i\in S} (\dim V^{i-1}/V^i-1)}\bigr).
\end{equation}
\end{prop}

\begin{proof}
As $g$-stable flags are both $g_s$-stable and $g_n$-stable, we may replace $g$ with either $g_s$ or $g_n$ (which cannot both be central since $g$ is not)
and assume without loss of generality that $g$ is either semisimple or nilpotent.

Since the set of $g$-stable flags with respect to $S$ is the union, as $T$ ranges over subsets of $S$, of the set of strictly $T$-stable flags,
it suffices to prove the bound (\ref{flag-est}) for the probability that a random flag is strictly $g$-stable.

We use induction on $m$.  Assume $m\ge 3$, and assume the proposition is known for flags of length $<m$, in particular flags of length $2$.  Thus, the probability
that a random $a$-dimensional subspace $W$ is strictly $g$-stable with respect to $S$ is $O(q^{1-n})$, $O(q^{-a})$, $O(q^{a-n})$, or $O(q^{-1})$ respectively if
$S$ is $\emptyset$, $\{1\}$, $\{2\}$, or $\{1,2\}$.  Applying this to $W := V^1$, the condition that $V^{\bullet}$ is strictly $g$-stable with respect to $S$ means that 
$g$ restricts to a (possibly central) endomorphism $g_1$ of $V^1$, and both of the following are true
\begin{enumerate}
\item The flag $V^2\supset \cdots\supset V^m = \{0\}$ in $V^1$ is strictly $g_1$-stable with respect to $S_1 := \{i\in \Z^{>0}\mid i+1\in S\}$.
\item The action of $g$ on $V/V^1$ is scalar if and only if $1\not\in S$.
\end{enumerate}
There are two cases to consider, according to whether $g_1$ is central or not.
In the central case,  $S_1$ must be empty.  This means that the flag $V^1\supset \{0\}$ in $V$ is strictly
$g$-stable with respect to $S$ (which is a subset of $\{1\}$).  By the induction hypothesis,
the probability of this occurring is $O(q^{1-n})$ or $O(q^{-a_1})$ as $S$ is empty or equal to $\{1\}$.
This is exactly the bound (\ref{flag-est}) in this case.

The probability that $g$ stabilizes $V^1$ but does not act on it as a scalar is $O(q^{-1})$ or $O(q^{a_1-n})$, depending on whether $1\in S$ or not.
The probability that (1) holds conditional on $g$ stabilizing $V^1$ and acting as the non-scalar endomorphism $g_1$
is 
$$O\bigl(q^{1-a_1+\sum_{i\in S_1}(\dim V^i/V^{i+1}-1)}\bigr) = O\bigl(q^{1-a_1+\sum_{i\in S\setminus\{1\}}(\dim V^{i-1}/V^i-1)}\bigr),$$
so multiplying by $O(q^{-1})$ or $O(q^{a_1-n})$, we obtain the bound (\ref{flag-est}).

It remains to prove the base case, $m=2$.  Assume first that $g$ is semisimple.  We proceed by induction on $n$.
If $V$ admits  a decomposition $V'\oplus V''$ into $g$-stable subspaces with no irreducible factor in common,
then every $g$-stable subspace $W\subset V$ is of the form $W'\oplus W''$, where $W'\subset V'$ and $W''\subset V''$.
Moreover, if the flag $W\supset \{0\}$ of $V$ is $g$-stable for $S$, the same is true of the flags $W'\supset \{0\}$ and $W''\supset \{0\}$ of $V'$ and $V''$ respectively.

We denote by $n'$ and $n''$ the dimensions of $W'$ and $W''$ and likewise by $a'$ and $a''$ the dimensions of $W'$ and $W''$.
Now,
$$\dim \Grass(n,a) - \dim \Grass(n',a')\times \Grass(n'',a'') = a'(n''-a'') + a''(n'-a') > 0$$
since $a,n',n'',n-a>0$ implies that each of the following conditions is impossible:
\begin{align*}
a'=a''&=0\\ a'=n'-a'&=0\\ n''-a''=a''&=0\\ n''-a''=n'-a'&=0.
\end{align*}
Thus, the probability that a random subspace $W$ decomposes as $W'\oplus W''$ is $O(q^{-1})$, and fixing $a',a'',n',n''$
and using induction, we conclude that the probability that $W\supset \{0\}$ is $g$-stable for $S$ is
\begin{align*}
O(q^{-1} q^{1-n'+\sum_{i\in S} (\dim {V'}^{i-1}/{V'}^i-1)} &q^{1-n''+\sum_{i\in S}(\dim  {V''}^{i-1}/{V''}^i-1)}) \\
&\qquad= O(q^{1-n+ \sum_{i\in S} (\dim V^{i-1}/V^i-1)}).
\end{align*}

We may therefore assume the action of $g$ on $V$ is isotypic, associated to a polynomial $f\in F$.  Since $g$ is not scalar, $d(f)\ge 2$.  The action of $g$ endows $V$ with the structure of $\F_{q^{d(f)}}$-vector space of dimensions $n/d(f)$.
For $W$ to be $g$-stable means just that it is an $\F_{q^{d(f)}}$-subspace.  The probability of this is $O\bigl(q^{(1/d(f)-1)a(n-a)}\bigr)$.
For $W\supset \{0\}$, regarded as a flag of $\F_q$-subspaces of $V$, to be $g$-stable with respect to $S$ implies that regarded as a flag of $\F_{q^{d(f)}}$-subspaces of $V$
it is also $g$-stable with respect to $S$.  By the induction hypothesis, the probability of this is
\begin{multline*}
O\Bigl(q^{(1/d(f)-1)a(n-a)}q^{d(f)\bigl(1-n/d(f)+\sum_{i\in S} (\dim_{\F_{q^{d(f)}}} V^{i-1}/V^i-1)\bigr)}\Bigr) \\
=O\bigl(q^{(1/d(f)-1)a(n-a)}q^{d(f)-1}q^{1-n+\sum_{i\in S} (\dim V^{i-1}/V^i-d(f))}\bigr).
\end{multline*}
Now $d(f)$ divides $a$ as well as $n$, so $2\le d(f) \le n/2$.  Thus,
$$d(f)-1 < \frac{n-1}2 \le \frac{a(n-a)}2 \le (1-1/d(f)) a(n-a),$$
so this implies (\ref{flag-est}).

Finally, we consider the case that $g$ is nilpotent.  If $S=\{1,2\}$, there is nothing to prove.  If $W\supset \{0\}$ is $g$-stable with respect to $\{1\}$, then $W$ is a subspace of $V$ on which $g$ acts as a scalar,
which must be $0$.  Thus, $W\subset \ker g$.  The probability that a random $a$-dimensional subspace of $V$ lies in any given space of dimension $\le n-1$ is $O(q^{-a})$, which gives (\ref{flag-est}).  If $W\supset \{0\}$ is $g$-stable with respect to $\{2\}$, then $g$ annihilates $V/W$, which means that $\im g\subset W$.  The probability that a random $a$-dimensional subspace of $V$ contains any specified non-zero vector is 
$O(q^{a-n})$, which again gives (\ref{flag-est}).  Finally, the probability that $\im g\subset W\subset \ker g$ is zero unless $g^2=0$ and $\dim \im g\le a\le \dim \ker g$, in which case it is
\begin{align*}
&\frac{|\Grass(\dim \ker g-\dim \im g,a - \dim \im g)|}{|\Grass(n,a)|}  \\ & \qquad\qquad\qquad= O(q^{a(n-\dim \ker g - \dim \im g) - \dim \im g\dim \ker g}).
\end{align*}
As $g^2 = 0$, $\dim \ker g+\dim \im g = n$, so we get a bound $O(q^{-\dim\ker g(n-\dim \ker g)})$ which in turn is bounded above by $O(q^{1-n})$, giving (\ref{flag-est}).

\end{proof}

We define the \emph{dimension} $\dim \tau$ of a type $\tau\in T_n$ to be
$$\dim \tau = \sum_{\lambda\in \Lambda} |\tau(\lambda)|.$$
Our definition is motivated by the following lemma:

\begin{lem}
\label{char-count}
For all $\tau\in T_n$,
$$|\Irra_\tau G_n| = O(q^{\dim\tau}).$$
\end{lem}

\begin{proof}
Every $\chi\in \Irra_\tau G_n$ determines a function $c\colon \Sigma\to \Lambda$ with the property that for $\lambda \neq \emptypart$, if $\tau(\lambda) =  1^{a_1}2^{a_2}\cdots$, then for all positive integers $s$,
there are exactly $a_i$ degree-$s$ simplices $\sigma\in c^{-1}(\lambda)$.  As $c$ is determined by this collection of data, it suffices to prove that the number of ways to choose $a_1$ $1$-simplices, $a_2$ $2$-simplices, and so on, is 
$$O(q^{a_1 + 2a_2 + \cdots}) = O(q^{|\tau(\lambda)|}).$$
Since every $s$-simplex is determined by a (generally non-unique) element $k\in \Z/p^s\Z$, it follows that there are  $O(q^{s a_s})$ ways in which to choose $a_s$ $s$-simplices.
\end{proof}

\begin{prop}
\label{parabolic}
If for every positive integer $n$, every primary type $\tau\in T_n$, every character $\chi$ of $G_n$ of type $\tau$, and every non-central element $g\in G_n$, we have
\begin{equation}
\label{key-bound}
\frac{\chi(g)}{\chi(1)} = O(q^{1-\dim\tau}),
\end{equation}
where the implicit constant does not depend on $\tau$, $\chi$, or $g$,
then for all $\tau\in T_n$ and every non-central element $g\in G_n$, we have
$$\sum_{\chi\in \Irra_\tau G_n} \frac{\chi(g)}{\chi(1)} = O(q).$$
\end{prop}

\begin{proof}
If $\tau$ decomposes as a product $\tau_1\cdots\tau_m$ of primary types of degrees $n_1,\ldots,n_m$, then
$$\sum_{\chi\in \Irra_\tau G_n} \frac{\chi(g)}{\chi(1)} = \frac 1{\deg \tau}\sum_{\chi_1\in \Irra_{\tau_1}G_{n_1}} \cdots\sum_{\chi_m\in \Irra_{\tau_m}G_{n_m}}   \chi_1\circ\cdots\circ\chi_m(g).$$
By (\ref{Hall}),
\begin{equation}
\begin{split}
\label{flag-sum}
\frac{\chi_1\circ\cdots\circ\chi_m(g)}{\deg\tau}  &= \frac{\chi_1\circ\cdots\circ\chi_m(g)}{|\cF_{n_1,\ldots,n_m}|\deg\tau_1\cdots\deg\tau_m} \\
             &= \frac{\sum_{(c_1,\ldots,c_m)}H_{c_1,\ldots,c_m}^g\frac{\chi_1(c_1)}{\chi_1(1)}\cdots\frac{\chi_m(c_m)}{\chi_m(1)}}{|\cF_{n_1,\ldots,n_m}|},
\end{split}
\end{equation}
where $\cF_{n_1,\ldots,n_m}$ denotes the set of flags $V^\bullet$ in $V$ with $\dim V^{i-1}/V^i = n_i$.
For $S\subset \{1,2,\ldots,m\}$, let $X_{m,S}$ denote the set of $m$-tuples $(c_1,\ldots,c_m)$ of conjugacy classes in $G_n$ such that $c_i$ is central if and only if $i\not\in S$.
Let $\cF_{n_1,\ldots,n_m}(g,S)$ denote the set of $V^\bullet$ which are strictly $g$-stable with respect to $S$.  Then
\begin{equation}
\label{S-decomp}
\begin{split}
\biggm|\sum_{(c_1,\ldots,c_m)}H_{c_1,\ldots,c_m}^g&\frac{\chi_1(c_1)}{\chi_1(1)}\cdots\frac{\chi_m(c_m)}{\chi_m(1)}\biggm| \\
&\le\sum_S  |\cF_{n_1,\ldots,n_m}(g,S)|
\max_{(c_1,\ldots,c_m)\in X_{m,S}}\biggm| \frac{\chi_1(c_1)}{\chi_1(1)}\cdots\frac{\chi_m(c_m)}{\chi_m(1)} \biggm|.
\end{split}
\end{equation}

By hypothesis, for $(c_1,\ldots,c_m)\in X_{m,S}$,
$$\frac{\chi_1(c_1)}{\chi_1(1)}\cdots\frac{\chi_m(c_m)}{\chi_m(1)} = O(\prod_{i\in S}q^{1-\dim\tau_i}).$$
By Proposition~\ref{flag-prob},
$$\frac{|\cF_{n_1,\ldots,n_m}(g,S)|}{|\cF_{n_1,\ldots,n_m}|} = O(q^{1-n+\sum_{i\in S}(n_i-1)}).$$
Combining this with (\ref{flag-sum}) and (\ref{S-decomp}), we obtain

\begin{equation}
\label{one-char}
\biggm| \frac{\chi_1\circ\cdots\circ\chi_m(g)}{\deg\tau} \biggm| = \sum_S   O\bigl(q^{\sum_{i\in S}(1-\dim \tau_i)+1-n+\sum_{i\in S} (n_i-1)}\bigr).
\end{equation}
By Lemma~\ref{char-count},
$$\Bigm|\prod_{i=1}^m \Irra_{\tau_i} G_{n_i}\Bigm| = O(q^{\sum_{i=1}^m \dim\tau_i}),$$
so as $\sum_{i\in S} n_i\le n$,
$$\sum_{\chi\in\Irra_\tau G_n}\frac{\chi(g)}{\chi(1)} = O(q).$$

\end{proof}

\section{Green polynomials}
In this section, we assemble some basic notation and facts related to Green polynomials.

If $\lambda$ is a partition of $m$, we denote by $\lambda'$ the conjugate partition and by $n_\lambda$ the sum of $x(x-1)/2$ as $x$ ranges over the parts of $\lambda'$.

\begin{lem}
\label{Q}
If $|\lambda|=m$, then
$$n_\lambda =\begin{cases}
m(m-1)/2&\text{if $\lambda = 1^m$},\\
(m-1)(m-2)/2 & \text{if $\lambda = 1^{m-2}2$}, 
\end{cases}$$
Otherwise,
$$n_\lambda \le (m-2)(m-3)/2+1.$$
\end{lem}

\begin{proof}
If $\lambda'_1\ge \ldots\ge \lambda'_r$ denote the sizes of the parts of $\lambda'$, then
$$n_\lambda = \sum_i \Bigl(\frac {(\lambda'_i)^2}2 - \frac {\lambda'_i}2\Bigr) = -\frac m2 + \frac{\sum (\lambda'_i)^2}2.$$
If $\lambda' = m$ or $\lambda' = 1(m-1)$, we obtain the asserted value.  If $\lambda'_1=1$, then $n_\lambda = 0$.
Otherwise, 
$$\sum_i (\lambda'_i)^2 \le (\lambda'_1)^2+(m-\lambda'_1)^2 \le (m-2)^2+4,$$
so we get the asserted bound.
\end{proof}

In \cite{Green}, Green defined for each pair $(\rho,\lambda)$ of partitions with $|\rho|=|\lambda|$ a polynomial $Q_\rho^\lambda(t)\in\Z[t]$.
Green's polynomials satisfy
$$\dim Q_\rho^\lambda(t)\le n_\lambda.$$
For $\lambda=1^m$, they are given explicitly as follows:
$$Q_{1^{r_1}2^{r_2}\cdots}^{1^m}(t) = \frac{\phi_m(t)}{\prod_i (1-t^i)^{r_i}},$$
where $\phi_m(t) = (1-t)(1-t^2)\cdots(1-t^m)$.  A.\ O.\ Morris \cite{Morris} proved:
$$Q_{1^{r_1}2^{r_2}\cdots}^{1^{m-2}2}(t) = \frac{\phi_{m-2}(t)((r_1-1)t^m-r_1t^{m-1}+1)}{\prod_i (1-t^i)^{r_i}}.$$

Every partition $\rho\in \Lambda_n$ determines up to conjugation an element $a_\rho$ in the symmetric group $S_n$, and we denote
by $z_\rho$ the order of the centralizer of $a_\rho$.  For each positive integer $s$, there is a unique endomorphism 
$p_s\colon \Lambda\to \Lambda$ defined on the generators of $\Lambda$ by $p_s(r) = rs$ for all $r\in \Z^{>0}$.  In particular, $p_s(\Lambda_v)\subset \Lambda_{sv}$.

\begin{prop}
\label{cancel}
For all positive integers $s$ and $v$ with $n=sv$,
\begin{equation}
\label{identity}
\sum_{\rho\in p_s(\Lambda_v)} \frac{Q_{\rho}^{1^n}(t)}{z_\rho} = \frac{\phi_n(t)}{s^v\phi_v(t^s)},
\end{equation}
and if $s>1$, then
\begin{equation}
\label{transvection}
\sum_{\rho\in p_s(\Lambda_v)} \frac{Q_{\rho}^{1^{n-2}2}(t)}{z_\rho} = \frac{(1-t^n)\phi_{n-2}(t)}{s^v\phi_v(t^s)}.
\end{equation}

\end{prop}

\begin{proof}

When $s=1$, (\ref{identity}) is just \cite[Lemma~5.2]{Green}.  For general $s$ and $\mu = 1^{m_1} 2^{m_2}\cdots\in \Lambda_v$,
\begin{align*}
Q_{p_s(\mu)}^{1^n}(t) = Q_{s^{m_1}(2s)^{m_2}\cdots}^{1^n}(t) = \frac{\phi_n(t)}{\prod_i (1-t^{is})^{m_i}} &= \frac{\phi_n(t)}{\phi_v(t^s)} \frac{\phi_v(t^s)}{\prod_i (1-(t^s)^i)^{m_i}} \\
& = \frac{\phi_n(t)}{\phi_v(t^s)}Q_\mu^{1^v}(t^s),
\end{align*}
and
\begin{align*}
z_{p_s(\mu)} = z_{s^{m_1}(2s)^{m_2}\cdots} &= s^{m_1} (2s)^{m_2}\cdots m_1! m_2! \cdots  \\
&= s^{m_1+2m_2+\cdots} 1^{m_1} 2^{m_2}\cdots m_1! m_2! \cdots = s^v z_\mu,
\end{align*}
which gives (\ref{identity}) in general.

For $s\ge 2$, $p_s(\mu)$ has no parts of size $1$, so for all $\mu\in \Lambda_v$,
$$Q_{p_s(\mu)}^{1^{n-2} 2}(t) = (1-t^{n-1})^{-1} Q_{p_s(\mu)}^{1^n}(t),$$
so (\ref{transvection}) follows from (\ref{identity}).
\end{proof}

Note in particular that the degree of the right hand size of (\ref{transvection})
is 
$$n + \frac{(n-2)(n-1)}2 - s \frac{v(v+1)}2 = \frac{n^2-(v+2)n+2}2.$$

\section{Primary character estimates}

\begin{prop}
\label{primary-case}
Suppose that for every pair $s,v$ of positive integers with $n=sv$, all $s$-simplices $\{k,\ldots,kq^{s-1}\}$, and all non-central $g\in G_n$,
$$\frac{I_s^k[v](g)}{I_s^k[v](1)} = O(q^{1-s}).$$
Then (\ref{key-bound}) holds for every primary type $\tau\in T_n$ and every character $\chi\in \Irra_\tau G_n$.
\end{prop}

\begin{proof}
Let $\chi = (-1)^{n-v}I_s^k[\lambda]$ be any primary character of $G_n$.  For $v_1,\ldots,v_r$ positive integers summing to $v= |\lambda|$, we set
$$\chi(v_1,\ldots,v_r) := I_s^k[v_1]\circ\cdots\circ I_s^k[v_r].$$
Equation (\ref{flag-sum}), inequality (\ref{S-decomp}), and estimate (\ref{one-char}) still hold for $\chi(v_1,\ldots,v_r)$, with $\tau_i=\tau=s$ for all $i=1,2,\ldots,r$, so 
$$\frac{\chi(v_1,\ldots,v_r)(g)}{\chi(v_1,\ldots,v_r)(1)} = O(q^{1-rs}).$$
Now $\chi$ can be expressed as a linear combination of  characters of the form $\chi(v_1,\ldots,v_r)$
where $v_1+\cdots+v_r = |\lambda|$, and the coefficients depend only on $\lambda$ and the $v_i$, not on $q$ \cite[p. 437]{Green};
the coefficients come from the expression of the Schur function of $\lambda$, $s_\lambda$, as a linear combination of 
products $s_{v_1}\cdots s_{v_r}$.  By \cite[I~\S6]{Macdonald}, these coefficients are non-zero if and only if $v_1\cdots v_r\succeq \lambda$ in the partial
order of partitions of $n$.  Thus, each has degree $\le \chi(1)$.  The proposition follows.
\end{proof}

Given a partition $\rho$ of $n$, we define a \emph{mode of substitition} $m$ of $\rho$ into $F$ to be a function $m\colon F\to \Lambda$ such that 
$$\prod_{f\in F} p_{d(f)}(m(f)) = \rho.$$
If $\varphi\colon F\to \Lambda$ is of degree $n$, and $c$ is the associated conjugacy class in $G_n$, we say $m$ is a mode of substitution into $c$ if
$$|m(f)| = |\varphi(f)|$$
for all $f\in F$.

\cite[Theorem 9]{Green} gives a formula for the ``principal parts'' $U_\rho$ of the character $I_s^k[v]$ which, combined with equations (18) and (19) in the same paper, gives 
an explicit formula for $I_s^k[v](c)$ for any conjugacy class $c$ of $G_n$.  The precise formula is not important to us, but it has the following general features.
It consists of an outer sum over partitions $\rho\in p_s(\Lambda_v)$.  For each $\rho$, there is an inner sum over modes of
substitution of $\rho$ into $c$.  The summand is a product of an $O(1)$ term and 
$$\prod_{f\in F} Q_{m(f)}^{\varphi(f)}(q^{d(f)}),$$
where the $O(1)$ factor does not depend on the mode if the semisimple part of $c$ is scalar.

For any partition  $\rho\in p_s(\Lambda_v)$, any mode of substitutions of $\rho$ into $c$, and any $f\in F$ for which $m(f)\neq \emptypart$, $d(f)|m(f)|$ is a positive integer divisible by $s$,
so $d(f)|\varphi(f)| = d(f)|m(f)| \ge s$.  Note that if there are at least two different elements $f_1,f_2\in F$ for which $m(f_i)$ is non-empty, it follows that no eigenvalue of an element of the conjugacy class $c$,
acting on $V$, can have eigenvalue multiplicity $\ge n-s$.

\begin{prop}
\label{Iskv}
For any $k$, $s$, and $v$ and any non-central $g$,
$$ \frac{I_s^k[v](g)}{I_s^k[v](1)} = O(q^{1-s}).$$
\end{prop}

\begin{proof}
We first suppose that $v\ge 4$.   By \cite[Lemma~7.4]{Green}, the degree of the irreducible character $\chi := (-1)^{n-v}I_s^k[v]$ is 
$$O(q^{\frac{n^2-vn}{2}}).$$
By a theorem of Bezrukavnikov-Liebeck-Shalev-Tiep \cite[Theorem~3.2]{BLST}, we have 
$$\chi(g) = O(\chi(1)^{1-\frac1{2n}}).$$
Thus 
$$\frac{\chi(g)}{\chi(1)} = O(q^{\frac{v-n}4}),$$
and $q^{\frac{v-n}4}\le q^{\frac{v-n}v} = q^{1-s}.$
We may therefore assume that $v\le 3$.

For $v=1$, an explicit formula for the character values is given by Green \cite[p.~431]{Green}.  The value is zero, unless $c$ is primary, in which case it can be written as a polynomial in $q$
with $O(1)$ coefficients, of degree at most
$$\frac{d(f) (n/d(f) - 1)(n/d(f)-2)}2 \le \frac{(n-1)(n-2)}2,$$
while $\chi(1) = |\phi_{n-1}(q)|$.  Thus
$$\frac{\chi(g)}{\chi(1)} = O(q^{1-n})$$
as desired.  We may therefore assume that $v$ is $2$ or $3$.  The proposition is trivial if $s=1$, so we may therefore assume $n\ge 4$.
By a theorem of David Gluck, $\chi(1)^{-1}|\chi(g)| = O(1/q)$, so we are justified in assuming $s\ge 3$.

Let $g_s g_u$ be the Jordan decomposition of $g$.  Suppose that $g_s$ is not a scalar.  Then we have an inclusion of centralizers $Z_{G_n}(g)\subset Z_{G_n}(g_s)$.
Since no eigenvalue of $g_s$ acting on $V$ has multiplicity $\ge n-s$, if $v=2$, the dimension of the centralizer is $n^2/2$, while if $v=3$, the dimension of the centralizer is at most $(n-s)^2 + s^2\le 5n^2/9.$  

By the centralizer estimate for characters, if $v=2$,
$$\chi(g) = O(q^{\frac{n^2}4}),$$
so
$$\frac{|\chi(g)|}{q^{1-n/2}\chi(1)} = O\bigl(q^{\frac{n^2}{4}-\frac{n^2-2n}2+\frac{n-2}2}\bigr)
= O\bigl(q^{\frac{-n^2+6n-4}4}\bigr).$$
For $n\ge 6$, this is $o(1)$.  

If $v=3$, we have
$$|\chi(g)| = O(q^{\frac{5n^2}{18}}),$$
so
$$\frac{|\chi(g)|}{q^{1-n/3}\chi(1)} = O\bigl(q^{\frac{5n^2}{18}-\frac{n^2-3n}2+\frac{n-3}3}\bigr) = O\bigl(q^{\frac{-4n^2+33n-18}{18}}\bigr),$$
which for $n\ge 9$ is $o(1)$.  

From now on, we assume $g_s$ is a scalar, so there exists a unique $f$ with $m(f)\neq\emptypart$, and $f$ is of degree $1$.
Moreover, $g$ non-central means $g_u\neq 1$.  We first consider the case that $g_u$ is not a transvection either, i.e., $\varphi(f) \neq1^{n-2} 2$; in particular, $n\ge 3$.
By Lemma~\ref{Q}, 
$$Q_{m(f)}^\varphi(f)(q^{d(f)}) \le O(q^{\frac{n^2-5n+8}2}).$$
Thus,%
$$\frac{\chi(g)}{q^{1-n/v}\chi(1)} = 
\begin{cases}
O(q^{\frac{6-2n}2})&\text{if $v=2$,}\\
O(q^{\frac{9-2n}3})&\text{if $v=3$.}
\end{cases}$$
If $v=2$ and $n\ge 4$ or $v=3$ and $n\ge 6$, the right hand side is $o(1)$.

All that remains is the case of transvections.  By Lemma~\ref{cancel},
$$\chi(g) = O\bigl(q^{\frac{n^2-(v+2)n+2}2}\bigr),$$
so
$$\frac{\chi(g)}{\chi(1)} = O(q^{1-n}).$$

\end{proof}

Combining Propositions \ref{parabolic}, \ref{primary-case}, and \ref{Iskv}, we obtain
\begin{thm}
\label{arith-flat}
If $g$ is any non-scalar element of $\SL_n(\F_q)$, then the number of pairs $(x,y)\in G_n^2$ such that $xyx^{-1}y^{-1} = g$ is 
$$|\xi_{G_n}^{-1}(g)| = O\bigl(q^{n^2+1}\bigr).$$
\end{thm}

\section{Geometric consequences}

\begin{thm}
The morphism $\xi_{\GL_n}\colon \GL_n^2\to \SL_n$ defined over any finite field $\F_q$ is flat at any point of $\GL_n^2$ which does not lie above the center of $\SL_n$.
\end{thm}

\begin{proof}
If $f\colon \uY\to \uX$ is any morphism of varieties, then every every irreducible component of every fiber of $f$ has dimension $\ge \dim \uY-\dim \uX$.  This applies to the commutator morphism
$\GL_n^2\to \SL_n$, so every irreducible component of every fiber has dimension $\ge 2n^2-(n^2-1) = n^2+1$.
On the other hand, if we work over the ground field $\F_q$ and fix a non-central element $g\in \SL_n(\F_q)$, then for every extension $K$ of $\F_q$, the set $K$-points of the fiber $\xi_{\GL_n}^{-1}(g)$
is the same as the fiber of the map of sets $\xi_{\GL_n(K)}\colon \GL_n(K)^2\to \SL_n(K)$.  In particular, taking $K$ to be a finite extension of $\F_q$, we have by Theorem~\ref{arith-flat} that 
$$|\xi_{\GL_n}^{-1}(g)(K)| = O(|K|^{n^2+1}).$$
By the Lang-Weil bound, this implies that every irreducible component of $\xi_{\GL_n}^{-1}(g)$ has dimension $\le n^2+1$.  Combining this with the lower bound on component dimension, we see that
every non-central fiber is purely of dimension $n^2+1$.  By miracle flatness \cite[Theorem~23.1]{Matsumura}, this implies that the restriction of $\xi_{\GL_n}$ to the open subset of points not lying over the center of $\SL_n$ is flat.
\end{proof}

\begin{cor}
The morphism  $\xi_{\SL_n}\colon \GL_n^2\to \SL_n$ defined over any finite field $\F_q$ is flat at any point of $\SL_n^2$ which does not lie above the center of $\SL_n$
\end{cor}

\begin{proof}
This question is invariant under base change, so we work over an algebraic closure $K$ of $\F_q$.  Then there is a surjective morphism with finite fibers
$$\GL_1\times \GL_1\times \xi_{\SL_n}^{-1}(g)\to \xi_{\GL_n}^{-1}(g)$$
defined by $(a,b,(c,d))\mapsto (ac,bd)$.  Thus, $\dim \xi_{\SL_n}^{-1}(g) = \xi_{\GL_n}^{-1}(g)-2 = n^2-1$ for all non-central $g$.  Again, by miracle flatness, $\xi_{\SL_n}$ is flat away from the central fibers.
\end{proof}

\begin{cor}
\label{main}
If $K$ is any field, and $\xi_{\SL_n}$ is defined over $K$, then $\xi_{\SL_n}$ is flat away from the central fibers.
\end{cor}

\begin{proof}
Since field extensions are faithfully flat, flatness is unaffected by field extensions, and the result follows immediately when $K$ is of positive characteristic.  On the other hand, suppose $K$ is of characteristic zero, and $\xi_{\SL_n}$ fails to be flat at some point $(a,b)$ with
$c := aba^{-1}b^{-1}$ non-central.  Then $\xi_{\SL_n}^{-1}(c)$ has dimension $\ge n^2$.  There exists an integral domain $A$ finitely generated  over $\Z$ such that $K$ is the fraction field of $A$, and $c$ belongs to the $\SL_n(A)$.
We  can regard $\SL_n$ as a scheme 
$$\SL_{n,A} := \Spec A[x_{11},\ldots,x_{nn}]/(\det (x_{ij})-1),$$
and $\xi_{\SL_n}$ extends to a morphism of schemes over $A$.  

By upper-semicontinuity of fiber dimension  \cite[Th\'eor\`eme~13.1.3]{EGA},
the set of points of $\SL_{n,A}$ over which the fiber of $\xi_{\SL_n}$ has dimension $\ge n^2$ is closed.  As schemes of finite type over $\Z$ are Jacobson \cite[Corollaire~10.4.6]{EGA}, the closed points in this set are Zariski dense, so in particular there is a non-central closed point.  The residue field of a closed point of $A$  is a finitely generated $\Z$-algebra and must therefore be finite, so it follows that there exists a finite field $\F_q$ and a non-central element of $\SL_n(\F_q)$
over which the $\xi_{\SL_n}$-fiber has dimension $\ge n^2$.  The contradiction proves the corollary.
\end{proof}

\section{Some central fibers}

We conclude by showing that some central fibers of $\xi_{\SL_n}$ also have dimension $n^2-1$.  This allows us to strengthen Corollary~\ref{main} when $n$ is prime.

\begin{prop}
Let $K$ be algebraically closed.
If $\zeta\in K^\times$ is a primitive $n$th root of unity, then
$$\xi_{\SL_n}^{-1}(\zeta I) \cong \PGL_n.$$
In particular, $\dim \xi_{\SL_n}^{-1}(\zeta I) = n^2-1$.
\end{prop}

\begin{proof}
If $(a,b)\in \xi_{\SL_n}^{-1}(\zeta I)(K)$, then $ab = \zeta ba$, and right-multiplying by $b^{-1}$, we see that the spectrum of $a$ is invariant under multiplication by $\zeta$, so the eigenvalues of $a$
can be written $-\alpha, -\zeta\alpha,-\zeta^2\alpha,\ldots,-\zeta^{n-1}\alpha$ for some $\alpha$.  In particular, since there are $n$ distinct eigenvalues, $a$ is diagonalizable.  On the other hand $1=\det(a)=(-1)^n\alpha^n\zeta^{n(n-1)/2}= 1$, so the set of eigenvalues of $a$ is exactly
$$\{-1,-\zeta,\zeta^2,\ldots,-\zeta^{n-1}\}.$$
The condition $ab = \zeta ba$ implies that $b$ maps the $-\zeta^i$-eigenspace of $a$ to the $-\zeta^{i+1}$-eigenspace.

Thus, for each $(a,b)$ in the fiber, there exists $g\in \GL_n(K)$ such that
$$g a g^{-1} = A := \begin{pmatrix}
-1&0&\cdots&0\\
0&-\zeta&\cdots&0\\
\vdots&\vdots&\ddots&\vdots\\
0&0&\cdots&-\zeta^{n-1}
\end{pmatrix}.$$
The condition $ab = \zeta ba$ now implies
$$gbg^{-1}=B := \begin{pmatrix}
0&0&\cdots&0&1\\
1&0&\cdots&0&0\\
0&1&\cdots&0&0\\
\vdots&\vdots&\ddots&\vdots&\vdots\\
0&0&\cdots&1&0
\end{pmatrix}.$$
Note that $(A,B)\in \xi_{\SL_n}^{-1}(\zeta I)(K)$, so the fiber is in fact non-empty.

It follows that $\xi_{\SL_n}^{-1}(\zeta I)$ forms a single orbit under the 
action of $\GL_n$ by simultaneous conjugation of both coordinates. The stabilizer consists of diagonal matrices which commute with $B$,
i.e., scalar matrices, so the orbit is $\PGL_n$.

\end{proof}

\begin{cor}
If $n$ is prime, $\xi_{\SL_n}$ is flat on the complement of the variety $\xi_{\SL_n}^{-1}(1)$ of commuting pairs.
\end{cor}

\end{document}